\documentclass[12pt]{article}

\usepackage{geometry}
\usepackage{amsthm}
\usepackage{pdfpages}
\usepackage[english]{babel}
\usepackage[utf8]{inputenc}
\usepackage{amsmath}
\usepackage{amsthm}
\usepackage{graphicx}
\usepackage[colorinlistoftodos]{todonotes}
\usepackage{wasysym}
\usepackage[numbers]{natbib}
\usepackage{babel}
\usepackage{bm}
\usepackage{graphicx}
\graphicspath{ {images/} }
\usepackage{natbib}
\usepackage{amsfonts}
\usepackage{ amssymb }
\usepackage{tabu}
\usepackage{mathtools}

\allowdisplaybreaks

\title{On the Number of Connected Components of Ranges of Divisor Functions}

\author{Nina Zubrilina}

\date{\today}
\begin{document}
\maketitle

\newtheorem{theorem}{Theorem}[section]
\newtheorem{lemma}[theorem]{Lemma}
\newtheorem*{claim}{Claim}
\newtheorem*{question}{Question}
\newtheorem{defn}[theorem]{Definition}
\newtheorem{remark}[theorem]{Remark}
\newtheorem{corollary}[theorem]{Corollary}
\newtheorem{prop}[theorem]{Proposition}
\newcommand{\diam}{\mathrm{diam}}
\newcommand{\edim}{\mathrm{edim}}
\newcommand\abs[1]{\left|#1\right|}
\newcommand{\tab}{\hspace*{1cm}}

\newcommand{\Prob}{\mathbb{P}}
\newcommand{\E}{\mathbb{E}}
\newcommand{\Q}{\mathbb{Q}}
\newcommand{\R}{\mathbb{R}}
\newcommand{\N}{\mathbb{N}}
\newcommand{\Z}{\mathbb{Z}}
\newcommand{\F}{\mathbb{F}}
\newcommand{\C}{\mathbb{C}}
\renewcommand{\F}{\mathbb{F}}
\newcommand{\Cl}{\operatorname{Cl}}
\newcommand{\eps}{\varepsilon}
\newcommand{\norm}[1]{\left\lVert#1\right\rVert}
\renewcommand{\d}{\partial}
\newcommand{\dint}{\mathrm{d}}
\newcommand*\Lapl{\mathop{}\!\mathbin\bigtriangleup}
\newcommand{\mat}[4] {\left(\begin{array}{cccc}#1 & #2\\ #3 & #4 \end{array}\right)}
\newcommand{\mattwo}[2] {\left(\begin{array}{cc}#1\\ #2 \end{array}\right)}
\newcommand{\pphi}{\varphi}
\newcommand{\eqmod}[1]{\overset{\bmod #1}{\equiv}}
\renewcommand{\O}[1]{\mathcal{O}_{#1}}
\newcommand{\bigO}[1]{\mathcal{O}\left( #1 \right)}
\newcommand{\fr}[1]{\mathfrak{#1}}
\newcommand{\p}{\mathfrak{p}}
\newcommand{\ord}{\operatorname{ord}}
\newcommand{\Li}{\operatorname{Li}}
\newcommand{\Ei}[1]{\operatorname{Ei}\left(#1\right)}
\newcommand{\sigbar}{ \overline{\sigma_{-r}(\N)}}

\DeclarePairedDelimiter\ceil{\left\lceil}{\right\rceil}
\DeclarePairedDelimiter\floor{\lfloor}{\rfloor}

\newcommand{\thistheoremname}{}
\newtheorem{genericthm}[theorem]{\thistheoremname}
\newenvironment{namedprob}[1]
  {\renewcommand{\thistheoremname}{#1}%
   \begin{genericthm}}
  {\end{genericthm}}
\newtheorem*{genericthm*}{\thistheoremname}
\newenvironment{namedprob*}[1]
  {\renewcommand{\thistheoremname}{#1}%
   \begin{genericthm*}}
  {\end{genericthm*}}
\linespread{2}
\begin{abstract}
For $r \in \R, r> 1$ and $n \in \Z^+$, the divisor function $\sigma_{-r}$ is defined by $\sigma_{-r}(n) := \sum_{d \vert n} d^{-r}$. In this paper we show the number $C_r$ of connected components of $\overline{\sigma_{-r}(\Z^+)}$ satisfies $$\pi(r) + 1 \leq C_r \leq \frac{1}{2}\exp\left[\frac{1}{2}\dfrac{r^{20/9}}{(\log r)^{29/9}} \left( 1 + \frac{\log\log r}{\log r - \log\log r} + \frac{\bigO{1}}{\log r}\right) \right],$$ where $\pi(t)$ is the number of primes $p\leq t$. We also show that $C_r$ does not take all integer values, specifically that it cannot be equal to $4$. 
\end{abstract}
\section{Introduction}

For a complex number $c \in \C$, we define the divisor function $\sigma_{c}: \N \to \R$ by
$$\sigma_{c}(n): = \sum\limits_{d \vert n} d^{c},$$ where for us $\N := \Z^+ := \{1, 2, \ldots\}$.
In 1986, Laatsch \cite{laatsch} studied the range of $\sigma_{-1}(\N)$. Laatsch showed it is a dense subset of $[1, + \infty)$ and asked if it is equal to $\Q \cap [1, \infty)$, to which Weiner (\cite{weiner}) answered negatively by showing $\Q \cap [1, \infty) \setminus \sigma_{-1}(\N) $ is also dense in $[1, + \infty)$. He asked what can be said about $\overline{\sigma_{c}(\N)}$ --- the topological closure of $\sigma_{c}(\N)$ for $c \in \C$. For an arbitrary complex $c$, this set has a complex fractal-like structure, which Defant studied in \cite{defantComplex} and \cite{Colinnew}.

A special case of the problem is $\sigbar$ for $r  \in \R, r > 1$. It is immediate that the range of this function is a subset of the interval $\left[1, \zeta(r)\right]$, where $$\zeta(r) := \sum\limits_{n \in \N} \dfrac{1}{n^r}$$ is the Riemann zeta function.
The divisor function is multiplicative, with 
$$\sigma_{-r}(p^\alpha) = 1 + \dfrac{1}{p^r} + \cdots + \dfrac{1}{p^{r\alpha}} = \dfrac{p^{(\alpha + 1)r }-1}{p^{(\alpha + 1)r}- p^{\alpha r}}.$$ We denote the $m$th prime number with $p_m$. We say $p_m$ is \emph{$r$-mighty} if 
$$1 + \dfrac{1}{p^r_m} > \prod\limits_{t= m+1}^\infty \dfrac{1}{1 - p_t^{-r}}.$$

We use $C_r$ to denote the number of disjoint intervals of $\sigbar$. Defant shows [\cite{Colinpaper}, Theorem 2.1] that $C_r=1$ if and only if no primes are $r$-mighty. He proves that this happens if and only if $r \leq \eta \approx 1.8877909$, where $\eta$ is a constant we will call the Defantstant. The Defantstant is the unique real number in $\left[1, 2\right]$ satisfying 
$$\dfrac{2^\eta}{2^\eta - 1}\dfrac{3^\eta+1}{3^\eta - 1} = \zeta(\eta).$$ In 2015, Sanna  \cite{Sannapaper} provides an algorithm to compute $\sigbar$ for a given $r$ and also shows that the number of connected components of $\sigbar$ is always finite. In this paper, we give effective bounds on $C_r$ from above and below. 

For a prime $p$, let $\sigma_{-r}(p^\infty) = \displaystyle{\lim_{n \to + \infty}}\sigma_{-r}(p^n) = \dfrac{1}{1 - p^{-r}}.$ We use $P_r $ to denote the largest $r$-mighty prime and set $P_r = 0$ if there are no $r$-mighty primes. We denote the number of $r$-mighty primes by $N_r$.

The paper is structured as follows. In the first section, we prove that for any positive constant $w$ with $w < 11/9$ , for all sufficiently large $r$, $$r < P_r  \leq \left(\dfrac{r}{w \log r}\right)^{20/9}, $$
$$ \pi(r) < N_r  \leq \pi\left(\dfrac{r^{20/9}}{(w \log r)^{20/9}}\right),$$ where $\pi(x)$ is the number of primes less or equal to $x$. In Section $2$, we use these bounds to deduce bounds for $C_r$, showing that
$$\pi(r) + 1 \leq C_r \leq \frac{1}{2}\exp\left[\frac{1}{2}\dfrac{r^{20/9}}{(\log r)^{29/9}} \left( 1 + \frac{\log\log r}{\log r - \log\log r} + \frac{\bigO{1}}{\log r}\right) \right].$$  Lastly, in Section $3$ we show that $C_r$ can never be equal to $4$.

\section{Number of Mighty Primes}
As mentioned earlier, Defant shows [\cite{Colinpaper}, Theorem 2.1] that $C_r=1$ if and only if there are no $r$-mighty primes. In this section, we prove some theorems about $r$-mighty primes which will allow us to obtain further bounds on $C_r$.
\begin{theorem}\label{Mr}
Let $w$ be a positive constant with $w < 11/9$. Then all primes $Q < r$ are $r$-mighty, and for all sufficiently large $r$ (the implicit constant is independent of $w$), all primes $Q>\dfrac{r^{20/9}}{(w \log r)^{20/9}}$ are not $r$-mighty. 
\end{theorem}

\begin{corollary}\label{Mrt}
We have 
$$P_r > r, N_r\geq \pi(r-1)$$ and for all sufficiently large $r$ (the implicit constant is again independent of $w$), 
 $$ P_r  \leq \left(\dfrac{r}{w \log r}\right)^{20/9}, N_r  \leq \pi\left(\dfrac{r^{20/9}}{(w \log r)^{20/9}}\right).$$
\end{corollary}

We will need some lemmas to prove the main results of the section. Let $r > \eta$, $m\in \N$, and $Q := p_m$. (We take $r > \eta$ since otherwise $P_r = N_r = 0$, as we know from \cite{Colinpaper}.)

\begin{defn}

For every $k, m \in \N$, let 
$$S_{k, m}(r): = \sum (p_{i_1}\cdots p_{i_k})^{-r},$$ where the sum is taken over all integers $i_1, \ldots, i_k$ such that 
$m< i_1 < \cdots < i_k$.
\end{defn}

\begin{lemma}\label{rephrased}
Let $r > \eta$. If $S_{1, m}(r) - S_{1, m}(r)^2/2 > \dfrac{1}{Q^r+1}$, then $Q$ is not $r$-mighty. Also, if $S_{1, m}(r) < \dfrac{1}{Q^r+1}$, then $Q$ is $r$-mighty.
\end{lemma}
\begin{proof}
By definition, 
\begin{align}
Q \text{ is $r$-mighty} &\iff (1 + Q^{-r}) > \prod\limits_{t = m+1}^\infty \dfrac{1}{1 - p_t^{-r}}  \nonumber\\
&\iff 1 - \dfrac{1}{Q^r + 1} < \prod\limits_{t = m+1}^\infty \left(1 - p_t^{-r}\right)  \nonumber\\
&\iff \dfrac{1}{Q^r + 1} >1 -  \prod\limits_{t = m+1}^\infty \left(1 - p_t^{-r}\right). \label{reformulated} 
\end{align}
We claim that 
\begin{align}\label{nuggets}
S_{1, m}(r) - S_{1, m}(r)^2/2 \leq 1 -  \prod\limits_{t = m+1}^\infty \left(1 - p_t^{-r}\right) \leq S_{1, m}(r).
\end{align}
 Clearly, this claim implies the statements of the lemma. \\
We begin by observing that 
$$\prod\limits_{t = m+1}^\infty \left(1 - p_t^{-r}\right) = 1 - S_{1, m}(r) + S_{2, m}(r) -S_{3, m}(r) +  \cdots.$$
Now, $S_{k, m}(r) > S_{k+1, m}(r).$ Indeed, we see that $\zeta(\eta)-1 < \zeta(2) - 1 = \frac{\pi^2}{6} - 1 < 1$, and thus
$$S_{1, m}(r) = \sum\limits_{t =m+1}^\infty p_t^{-r} \leq -1 +   \sum\limits_{n =1}^\infty n^{-r}  = \zeta(r) - 1 \leq \zeta(\eta)-1 <1.$$
It follows that 
$$S_{k+1, m}(r) \leq \sum\limits_{t =m+1}^\infty p_t^{-r} S_{k, m}(r) = S_{1, m}(r)\cdot S_{k, m}(r) < S_{k, m}(r)$$ for all $k \in \N$.   Note also that 
$$S_{2, m}(r) = \sum\limits_{i >j > m} (p_ip_j)^{-r} = \dfrac{1}{2}\left(\sum\limits_{i, j > m} (p_ip_j)^{-r} -  \sum\limits_{i > m} \left(p_i^2\right)^{-r} \right) < \dfrac{S_{1, m}(r)^2}{2}.$$
Thus, since  $1 -  \prod\limits_{t = m+1}^\infty \left(1 - p_t^{-r}\right) = S_{1, m}(r) - S_{2, m}(r) + S_{3, m}(r) - \cdots, $ we see that
$$ S_{1, m}(r) - S_{1, m}(r)^2/2 \leq S_{1, m}(r) - S_{2, m}(r) \leq 1 -  \prod\limits_{t = m+1}^\infty \left(1 - p_t^{-r}\right) \leq S_{1, m}(r). $$
This proves (\ref{nuggets}). Combining  (\ref{nuggets}) with (\ref{reformulated}), we get the statement of the lemma.
\end{proof}
We will be working with the following integral form of $S_{1, m}(r)$: 
\begin{lemma}\label{S1intform}
We have
$$S_{1, m}(r) = r\int_{Q}^\infty\dfrac{\pi(t) - \pi({Q})}{t^{r+1}}\dint t. $$
\end{lemma}
\begin{proof}
Using the Riemann-Stieltjes integral for $S_{1, m}(r)$ and integration by parts, we get
\begin{align*}
S_{1, m}(r) = \sum\limits_{t =m+1}^\infty p_t^{-r} &= \int_{Q+}^\infty \dfrac{1}{t^r} \dint \pi(t) \\
& = \dfrac{\pi(t)}{t^r}\Big\vert_{Q+}^\infty + r\int_{Q}^\infty\dfrac{\pi(t)}{t^{r+1}}\dint t \\
&=  -\dfrac{\pi({Q})}{{Q}^r}+ r\int_{Q}^\infty\dfrac{\pi(t)}{t^{r+1}}\dint t  \\
&=  r\int_{Q}^\infty\dfrac{\pi(t) - \pi({Q})}{t^{r+1}}\dint t 
\end{align*}
as desired.
\end{proof}
We will now prove one of the two statements of Theorem \ref{Mr}.
\begin{theorem}\label{mainlower}
Let $r > 1$, $Q$ be a prime and suppose $Q<r$. Then $Q$ is $r$-mighty.
\end{theorem}
\begin{proof}
First, we check with a computer calculation that $2$ is $2$-mighty, and hence it is $r$-mighty for all $r > 2$ (see Lemma \ref{rpexists}). Hence it suffices to check the claim for $Q \geq 3$. 

Next, as we see from Lemma \ref{rephrased} and Lemma \ref{S1intform}, it is sufficient to show that for such $Q>2$, 
$$\frac{1}{Q^r + 1} >  r\int_{Q}^\infty\dfrac{\pi(t) - \pi({Q})}{t^{r+1}}\dint t .$$
Since $Q>2$ is prime, we have that $\pi(t) - \pi(Q) = 0$ when $t \in [Q, Q+2)$. \\ Thus:
\begin{align*}
 r\int_{Q}^\infty\dfrac{\pi(t) - \pi({Q})}{t^{r+1}}\dint t &=  r\int_{Q+2}^\infty\dfrac{\pi(t) - \pi({Q})}{t^{r+1}}\dint t\\
&\leq  r\int_{Q+2}^\infty\dfrac{t - Q}{t^{r+1}}\dint t\\
&=  r\int_{Q+2}^\infty\dfrac{1}{t^{r}}\dint t -  r\int_{Q+2}^\infty\dfrac{Q}{t^{r+1}}\dint t\\ 
&= \frac{r}{r-1}\frac{1}{(Q+2)^{r-1}} - \frac{Q}{(Q+2)^r}\\
&= \frac{1}{(Q + 2)^{r-1}}\left( \frac{1}{r - 1} + \frac{2}{Q + 2}  \right)\\
&\leq \frac{1}{(Q + 2)^{r-1}}\left( \frac{1}{Q - 1} + \frac{2}{Q + 2}  \right) \\
&=  \frac{1}{(Q + 2)^{r}} \frac{3Q}{(Q - 1)}.
\end{align*}
Hence, to check that $Q$ is $r$-mighty, it is sufficient to check that $Q^r + 1 < \dfrac{\left(Q+2\right)^r\left(Q-1\right)}{3Q}$. Multiplying both sides by $3Q/Q^r$, this inequality becomes
$$ 3Q\left( 1 + \frac{1}{Q^r}\right) < \left(Q-1\right)\left(1 + \frac{2}{Q}\right)^r.$$
Since we assumed $Q <r$, it suffices to show $ 3Q\left( 1 + \frac{1}{Q^Q}\right) < \left(Q-1\right)\left(1 + \frac{2}{Q}\right)^Q$
or, equivalently, 
$$ 3\left( 1 + \frac{1}{Q-1} \right)\left( 1 + \frac{1}{Q^Q}\right) < \left(1 + \frac{2}{Q}\right)^Q.$$
For $Q \geq 3$, the left-hand side decreases and the right-hand side grows, so it suffices to check this for $Q = 3$, which holds.
\end{proof}

We will need the following theorem and two technical lemmas presented below to prove the other part of Theorem \ref{Mr}.

\begin{theorem}[\cite{difference}]\label{pidiflower}
For large enough $x$ and $y > x^{11/20}$, 
$$\pi(x + y) - \pi(x) \gg \frac{y}{\log( x + y)},$$ where $f\ll g$ means $f\leq \O{}(g)$ --- i.e., that $f\leq c\cdot g$ as functions, where $c$ is a positive constant.
\end{theorem}

\begin{lemma}\label{integcalc}
For $K, M> 1$, we have
$$\int_{M}^\infty\dfrac{1}{t^{K}\log{t}}\dint t  = - \Ei{-(K-1)\log{M}},$$
where $\Ei{x} := {\displaystyle \int_{-\infty}^x \dfrac{e^{t}}{t} \ \dint t}$ is the exponential integral. 
\end{lemma}  
\begin{proof}
\begin{align*}
\int_{M}^\infty\dfrac{1}{t^{K}\log{t}}\dint t  &= \int_{M}^\infty \dfrac{1}{t^{K-1}\log{t}} \ \dint \log{t} \nonumber \\
&= \int_{\log{M}}^\infty \dfrac{e^{(1-K)u}}{u} \dint u  \\
&= (K-1)  \int_{\log{M}}^\infty \dfrac{e^{(1-K)u}}{(K-1)u} \dint u \\
&=  \int_{(K-1) \log{M}}^\infty \dfrac{e^{-v}}{v} \dint v  \\
&=- \Ei{-(K-1)\log{M}}. 
\end{align*}
This completes the proof.
\end{proof}

\begin{lemma}\label{Eibounds}
We have
$$ \dfrac{e^{-x}}{x}-  \dfrac{e^{-x}}{x^2} \leq -\Ei{-x} \leq\dfrac{ e^{-x}}{x}. $$ 
\end{lemma}
\begin{proof}
Due to  \cite[p. 229, 5.1.20]{Eiref}, we know that, for $x > 0$, 
$$ \dfrac{e^{-x}}{2}\log\left(1 + \dfrac{2}{x}\right) \leq -\Ei{-x} \leq e^{-x}\log\left(1 + \dfrac{1}{x}\right).$$ Note that $t - t^2/2 < \log (1 + t) < t $ for $t > 0$. Hence, 
$$ \dfrac{e^{-x}}{x}-  \dfrac{e^{-x}}{x^2}  \leq \dfrac{e^{-x}}{2}\log\left(1 + \dfrac{2}{x}\right) \leq -\Ei{-x} \leq e^{-x}\log\left(1 + \dfrac{1}{x}\right) \leq \dfrac{ e^{-x}}{x}. \qedhere$$
\end{proof}
We now apply these lemmas to bound $S_{1, m}(r)$ from below.  
\begin{lemma}\label{S1est}
For sufficiently large $r$ and $Q = p_m $, 
$$S_{1, m}(r) \gg \dfrac{1}{r(Q + Q^{11/20})^{r-1}\log Q}.$$
\end{lemma}
\begin{proof}
Set $Q_0 := Q(1 + Q^{-9/20})$ for convenience. Using Lemmas \ref{S1intform}, \ref{pidiflower} and  \ref{integcalc},  we get 
\begin{align*}
S_{1, m}(r) &= r\int_Q^\infty\dfrac{\pi(t) - \pi(Q)}{t^{r+1}}\dint t \\
&\geq r\int_{Q_0}^\infty\dfrac{\pi(t) - \pi(Q)}{t^{r+1}}\dint t  \\
&\gg  r\int_{Q_0}^\infty\dfrac{t - Q}{t^{r+1}\log t}\dint t \\
&= r\int_{Q_0}^\infty\dfrac{1 }{t^{r}\log t}\dint t - r\int_{Q_0}^\infty\dfrac{Q}{t^{r+1}\log t}\dint t \\
&=r\left[ -\Ei{-(r-1)\log(Q_0)} - Q(- \Ei{-r\log Q_0}) \right]\\
&\geq r\left[ -\Ei{-(r-1)\log(Q_0)} - Q_0(-\Ei{-r\log Q_0} )\right].
\end{align*}
Applying Lemma \ref{Eibounds}, we see that 
\begin{align*}
 &r\left[ -\Ei{-(r-1)\log(Q_0)} - Q_0(-\Ei{-r\log Q_0} )\right] \\
\geq \ &r\left[ \dfrac{1}{Q_0^{r-1}(r-1)\log Q_0}-\dfrac{1}{Q_0^{r-1}(r-1)^2 \log^2(Q_0)} - \dfrac{1}{rQ_0^{r-1}\log(Q_0)} \right] \\
= \ &r\left[ \dfrac{1}{Q_0^{r-1}r(r-1)\log Q_0}-\dfrac{1}{Q_0^{r-1}(r-1)^2 \log^2(Q_0)}\right] \\
= \ &\dfrac{1}{(r-1)Q_0^{r-1}\log Q_0}\left[1 -\dfrac{r}{(r-1)\log(Q_0)}\right] \\
\gg \ &\dfrac{1}{(r-1)Q_0^{r-1}\log Q_0}\ .
\end{align*}
Note that $\log Q_0 = \log Q + \log(1 + Q^{-9/11}) \ll \log Q, $ so $$\dfrac{1}{(r-1)Q_0^{r-1}\log Q_0}  \gg \dfrac{1}{(r-1)Q^{r-1}(1 + Q^{-9/11})^{r-1}\log Q},$$ which gives the desired result. 
\end{proof}

\begin{lemma}
Let $0<w <11/9$. For sufficiently large $r$, if $Q$ is prime and $Q> \left(\dfrac{r}{w \log r}\right)^{20/9}$, then $Q$ is not $r$-mighty. 
\end{lemma}
\begin{proof}
Since $S_{1, m}(r)$ tends to zero as $r$ goes to infinity, for sufficiently large $r$,  $$S_{1, m}(r) - S_{1, m}(r)^2/2 >S_{1, m}(r)/2 \gg S_{1, m}(r)\gg  \dfrac{1}{rQ^{r-1}(1 + Q^{-9/11})^{r-1}\log Q}$$
by Lemma \ref{S1est}. Hence, by Lemma \ref{rephrased}, there exists some positive constant $k$ such that $Q$ is not $r$-mighty for $Q$ satisfying 
$$Q^r  > kr Q^{r-1}(1 + Q^{-9/20})^{r-1}\log Q$$
or, equivalently,
\begin{align}\label{shosho}
\frac{Q}{\log Q} >  k r(1 + Q^{-9/20})^{r-1}.
\end{align}
 We claim that for large enough $r$, (\ref{shosho}) holds for $Q> \left(\dfrac{r}{w \log r}\right)^{20/9}.$ 
When $Q>e$, the left-hand side is increasing in $Q$ while the right-hand side is decreasing in $Q$. Therefore, it is sufficient to prove the inequality when $Q=\left(\dfrac{r}{w\log r}\right)^{20/9}$ and $r$ is sufficiently large. In this case, the inequality becomes
\begin{align}
&\left(\frac{r}{w \log r}\right)^{20/9} > \frac{20 k r }{9} (\log r - \log\log r - \log w)\left( 1 + \frac{w \log r}{r}\right)^{r-1}. \label{abc}
\end{align}
Note that $$\left( 1 + \dfrac{w \log r}{r}\right)^{r-1} = \exp\left[\log\left(1 + \dfrac{w\log r}{r}\right)\left(r-1\right)\right] < \exp\left[\frac{\left(r-1\right)w\log r}{r}\right] \leq r^w.$$
Hence, the right-hand side of (\ref{abc}) is less than a positive constant times $r^{(1 + w)} \log r$. In the meantime the left-hand side is a positive constant times $\dfrac{r^{20/9}}{(\log r)^{20/9}}$. Since by assumption $1 + w < 20/9 $, (\ref{abc}) must hold for large $r$.
\end{proof}
This proves that for any positive constant $w$ with $w < 11/9$, we have $P_r  <  \left(\dfrac{r}{w \log r}\right)^{20/9}$ for large $r$.

\section{Bounds on the Number of Intervals}
 Recall that $P_r$ is the largest $r$-mighty prime and that $N_r$ is the number of $r$-mighty primes. Furthermore, recall that for each prime $p$, we define $\sigma_{-r}(p^\infty)=\dfrac{1}{1-p^{-r}}$. In this section, we will estimate the number $C_r$ of connected components of $\overline{\sigma_{-r}(\N)}$ using the bounds on $P_r$ and $N_r$.
\begin{theorem}\label{mingaps}
We have
$$C_r \geq N_r + 1 \geq \pi(r) + 1.$$
\end{theorem}
\begin{proof}
Let $Q = p_m$ be an $r$-mighty prime, that is, $$1 + Q^{-r} > \prod\limits_{t = m+1}^\infty \dfrac{1}{1 - p_t^{-r}}.$$ Let $N \in \N$. Suppose $N$ has a prime divisor $q \leq Q.$ Then $$\sigma_{-r}(N) \geq 1 + q^{-r} \geq 1 + Q^{-r}.$$
On the other hand, suppose that all prime divisors of $N$ are larger than $Q$. Note that 
$$\prod\limits_{t = m+1}^\infty \dfrac{1}{1 - p_t^{-r}} = \prod\limits_{t = m+1}^\infty\left(1 + \dfrac{1}{p_{t}^r} +  \dfrac{1}{p_{t}^{2r}} + \cdots \right),$$ and if we expand the product we will get all possible terms of the form $(q_1^{\alpha_1}\cdots q_k^{\alpha_k})^{-r}$ with prime $q_i > Q$. Hence, in this case 
$$\sigma_{-r}(N) \leq  \prod\limits_{t = m+1}^\infty \dfrac{1}{1 - p_t^{-r}}.$$ 
Lastly, note that both $1 + Q^{-r}$ and $\prod\limits_{t = m+1}^\infty \dfrac{1}{1 - p_t^{-r}}$ themselves are in $\overline{\sigma_{-r}(\N)}$. Indeed, 
$$1 + Q^{-r} = \sigma_{-r}(Q) \in\sigbar$$ and 
\begin{align*}
 \prod\limits_{t = m+1}^\infty \dfrac{1}{1 - p_t^{-r}} = &\lim_{s \to \infty}  \prod\limits_{t = m+1}^{m + s} \left(1 + \dfrac{1}{p_{t}^r} +  \dfrac{1}{p_{t}^{2r}} + \cdots + \dfrac{1}{p_{t}^{sr}} \right)\\ = 
&\lim_{s \to \infty}\sigma_{-r}(p_{m+1}^s\cdots p_{m+s}^s) \in \sigbar.
\end{align*}
Hence, $\sigbar$ has a gap $\left( \prod\limits_{t = m+1}^\infty \dfrac{1}{1 - p_t^{-r}}, 1 + Q^{-r}\right)$ (where by a \emph{gap} of the closed set $K\subseteq\R$, we mean a bounded connected component of $\R\setminus K$). It follows that the total number of gaps of $\overline{\sigma_{-r}(\N)}$ is at least the number $N_r$ of $r$-mighty primes, so the number of connected components $C_r$ is at least $N_r + 1$. Lastly, by Corollary \ref{Mrt}, $N_r + 1 \geq \pi(r) + 1$. 
\end{proof}

In order to bound $C_r$ from above, we will use the algorithm of Sanna \cite{Sannapaper}. 
\begin{defn}
Define
$$\N_{j} = \{n  \in \N \ \vert \ n \text{ has no prime divisors less than } p_{j + 1}\}.$$
\end{defn}
Let $L_r \in \N$ be the index of the largest $r$-mighty prime (so $P_r = p_{L_r}$). In \cite{Sannapaper}, Sanna proved the following theorems.
\begin{theorem}[\cite{Sannapaper}, Lemma 2.3]\label{bananafish}
$$\overline{\sigma_{-r}(\N_{L_r})} = \left[1,    \prod\limits_{t = L_r+1}^\infty \dfrac{1}{1 - p_t^{-r}}\right].$$ 
\end{theorem}

\begin{theorem}[\cite{Sannapaper}, Lemma 2.4]\label{inductstep}
$$\overline{\sigma_{-r}(\N_{K})} = \bigcup_{i \in \Z_{\geq 0} \cup \{\infty\}} \sigma_{-r}(p_{K+1}^i) \cdot \overline{\sigma_{-r}(\N_{K+1})},$$ where we write $a \cdot X =\{ax \ \vert \ x\in X\}$ for a number $a$ and a set $X$. 
\end{theorem}

\begin{theorem}[\cite{Sannapaper},  Lemma 2.5]\label{structurethm}
Let $I = \left[a, b\right]$ be an interval $a > b > 1.$ Let $p $ be a prime. Let $t$ be the least non-negative integer such that $$\dfrac{\sigma_{-r}(p^{t+1})}{\sigma_{-r}(p^{t})} \leq \dfrac{b}{a}.$$ Then the following is a decomposition into disjoint closed intervals:
$$\bigcup_{i \in \Z_{\geq 0} \cup \{\infty\}} \sigma_{-r}(p^i) \cdot I = \bigcup_{0\leq  i < t} \sigma_{-r}(p^i) \cdot I \cup \left[a\sigma_{-r}(p^t), b\sigma_{-r}(p^\infty)\right].$$ 
\end{theorem}

With these three theorems, Sanna demonstrated a backwards induction algorithm to calculate $\overline{\sigma_{-r}(\N_{L_r-1})}, \overline{\sigma_{-r}(\N_{L_r-2})}, \ldots, \overline{\sigma_{-r}(\N_0)} = \overline{\sigma_{-r}(\N)}.$ The algorithm goes as follows: 
\begin{enumerate}
\item We know $\overline{\sigma_{-r}(\N_{L_r})} = \left[1,  \prod\limits_{t = L_r+1}^\infty \dfrac{1}{1 - p_t^{-r}}\right].$\\
\item Suppose we have calculated $\overline{\sigma_{-r}(\N_{K})} = \bigcup_{j \in J} I_j $ for $K \in \N$ and some index set $J$. For each $I_j = \left[a_j, b_j\right]$, let $t_j$ be the smallest $t \in  \Z_{\geq 0} $ such that $\dfrac{\sigma_{-r}(p_K^{t+1})}{\sigma_{-r}(p_K^{t})} \leq \dfrac{b_j}{a_j}$. 
\item Then $$\overline{\sigma_{-r}(\N_{K-1})} = \bigcup_{j \in J}\left( \bigcup_{0\leq  i < t_j} \sigma_{-r}(p_K^i) \cdot I_j \cup \left[a_j\sigma_{-r}(p_K^{t_j}), b_j\sigma_{-r}(p_K^\infty)\right]\right).$$ These intervals might be not pairwise disjoint, but there are still finitely many of them. 
\end{enumerate}

Using Sanna's result, we prove the following.
\begin{theorem}\label{crformula}
The number of connected components of $\overline{\sigma_{-r}(\N)}$is at most $\prod\limits_{i = 1}^{L_r}\left\lceil\dfrac{\log p_{L_{r}+1}}{\log{p_{i}}}\right\rceil.$
\end{theorem}
\begin{proof}
Let $$\ell = \max \overline{\sigma_{-r}(\N_{L_r}}).$$  
We proceed by the same backwards induction process to prove the following. \\
For $0\leq d \leq L_r,$ $\overline{\sigma_{-r}(\N_{L_r - d})}$ consists of at most $\prod\limits_{i = 1}^{d}\left\lceil\dfrac{\log p_{L_r+1}}{\log{p_{L_r + 1 - i}}}\right\rceil$ disjoint intervals, and each interval $\left[a, b\right]$ satisfies $a/b \geq \ell$. \\
For $d = 0$, by Sanna's result 
$$\overline{\sigma_{-r}(\N_{L_r - d})} = \left[1,    \prod\limits_{t = L_r+1}^\infty \dfrac{1}{1 - p_t^{-r}}\right].$$  Hence there is exactly one interval, and the ratio of its endpoints is exactly $\ell$.\\
Assume $d \geq 1$. Suppose $\overline{\sigma_{-r}(\N_{L_r - (d-1)})} $ is a union of $k \leq  \prod\limits_{i = 1}^{d-1}\left\lceil\dfrac{\log p_{L_r+1}}{\log{p_{L_r +1 - i}}}\right\rceil$ intervals $I_1, \ldots, I_k$ with endpoint ratios at least $\ell$. For simplicity let $p= p_{L_r - d +1}$. 
Recall that by Theorem \ref{inductstep}, 
$$\overline{\sigma_{-r}(\N_{L_r -d})}= \bigcup_{1 \leq j \leq k}  \bigcup_{i \in \Z_{\geq 0} \cup \{\infty\}} \sigma_{-r}(p^i) I_j.$$
The ratio of the endpoints of the interval $ \sigma_{-r}(p^i) I_j$ is the same as that of $I_j$, which is at least $\ell$. Also, note that the union of two intersecting intervals with endpoint ratio at least $ \ell$ is an interval with endpoint ratio at least $ \ell$ as well. Hence, if we take the union of all these intervals, the resulting set will be a union of  disjoint intervals which will also have endpoint ratios at least $ \ell.$ 

Now we bound the number of intervals. Let $I \in \{I_1, \ldots, I_k\}$. We want to bound the minimal $t$ such that 
\begin{align}\label{propz}
\dfrac{\sigma_{-r}(p^{t+1})}{\sigma_{-r}(p^{t})} \leq \ell.
\end{align}
Note that
$$\dfrac{\sigma_{-r}(p^{t+1})}{\sigma_{-r}(p^{t})} = \dfrac{1 + p^{-r} + \cdots + p^{-(t + 1)r}}{1 + p^{-r} + \cdots + p^{-tr}} = 1 + \dfrac{1}{p^r + p^{2r} + \cdots + p^{(t+1)r}} \leq 1 + \dfrac{1}{p^{(t+1)r}}.$$
Also, $\ell  \geq 1 + p^{-r}_{L_r+1}$, so for $t \geq \dfrac{\log p_{L_r+1}}{\log{p}}-1$, we have
\begin{align*}
(t + 1)\log{p} \geq  \log(p_{L_r + 1}) &\implies p^{r(t+1)} \geq p^{r}_{L_r + 1} \\
&\implies 1 + \dfrac{1}{p^{(t+1)r}} \leq 1 + \dfrac{1}{p^r_{L_r + 1}} \leq \ell.
\end{align*}
Hence the smallest $t$ satisfying (\ref{propz}) is at most  $\left\lceil\dfrac{\log p_{L_r+1}}{\log{p}}-1\right\rceil$. This implies that 
$ \bigcup_{i \in \Z_{\geq 0} \cup \{\infty\}} \sigma_{-r}(p^i) I$ is a union of at most $\left\lceil\dfrac{\log p_{L_r+1}}{\log{p}}\right\rceil$ intervals, and hence (recalling that $p = p_{L_r - d +1}$) we see that $\overline{\sigma_{-r}(\N_{L_r -d})}$ is a union of at most 
$$k\left\lceil \dfrac{\log p_{L_r+1}}{\log{p_{L_r - d }}}\right\rceil \leq \prod\limits_{i = 1}^{d}\left\lceil\dfrac{\log p_{L_r+1}}{\log{p_{L_r +1- i}}}\right\rceil$$ intervals. This completes the induction step. \newline
Thus, $\overline{\sigma_{-r}(\N_0)} = \overline{\sigma_{-r}(\N)}$ consists of at most $\prod\limits_{i = 1}^{L_r}\dfrac{\log p_{L_{r}+1}}{\log{p_{L_r + 1 - i}}} = \prod\limits_{i = 1}^{L_r}\dfrac{\log p_{L_{r}+1}}{\log{p_{i}}}$ intervals. \qedhere
\end{proof}

\begin{theorem}\label{pampampam}
With $L_r$ as above, we have
$$\prod\limits_{i = 1}^{L_r}\left\lceil\dfrac{\log p_{L_r+1}}{\log{p_{i}}} \right\rceil\leq \frac{1}{2}\exp \left[ \log 2 \frac{ p_{L_r+1}}{\log{p_{L_r+1}}} +  \bigO{ \dfrac{p_{L_r+1}}{\log^2 p_{L_r+1}}} \right]  $$
\end{theorem}

\begin{proof}
First, note that 
\begin{align*}
\prod\limits_{i = 1}^{L_r}\left\lceil\dfrac{\log p_{L_r+1}}{\log{p_{i}}} \right\rceil &\leq \prod\limits_{i = 1}^{L_r}\dfrac{\log p_{L_r+1}+\log{p_{i}}}{\log{p_{i}}} \\
&= \frac{1}{2}  \prod\limits_{i = 1}^{L_r+ 1}\dfrac{\log p_{L_r+1}+\log{p_{i}}}{\log{p_{i}}}\\
&=\frac{1}{2}\exp\left[  \sum_{i = 1}^{L_r+1} \log \log (p_i p_{L_r + 1}) - \sum_{i = 1}^{L_r+1} \log \log p_i \right].
\end{align*}
For simplicity, we put $S:=p_{L_r + 1} $ and estimate the exponent using the Riemann-Stieltjes integral. We have
\begin{align*}
 \sum_{i = 1}^{L_r+1} \log \log p_i p_{L_r + 1} &= \int_{2-}^{S+} \log \log (S x) \ \dint \pi(x)\\
& = \pi(x) \log \log (S x) \  \Big|_{2-}^{S+}  -  \int_{2-}^{S+} \frac{\pi (x)}{x \log (S x)} \ \dint x\\
&=\pi(S) \log \log S^2  - \int_{2}^{S} \frac{\pi (x)}{x \log (S x)} \ \dint x
\end{align*}
and similarly 
$$  \sum_{i = 1}^{L_r+1} \log \log p_i  = \int_{2 - }^{S+} \log \log (x) \ \dint \pi(x) = \pi(S)\log\log S  - \int_{2}^{S} \frac{\pi (x)}{x \log (x)} \ \dint x. $$
\begin{align*}
&\pi(S) \log \log S^2 - \int_{2}^{S} \frac{\pi (x)}{x \log (S x)} \ \dint x -  \pi(S)\log\log S + \int_{2}^{S} \frac{\pi (x)}{x \log (x)}\ \dint x  \\
= &\pi(S) \log \frac{\log S^2}{\log S} +   \int_{2}^{S} \frac{\pi (x)(\log (Sx) - \log x)}{x \log(Sx) \log (x)}    \ \dint x\\
=& \pi(S) \log 2 +  \int_{2}^{S} \frac{\pi (x)\log S }{x \log(Sx) \log (x)}    \ \dint x.
\end{align*}
We now estimate the remaining integral. Since $\pi(x) < \dfrac{2x}{\log x}$ for all $x \geq 2$ (see \cite{piconstbound}), 
\begin{align*}
  \int_{2}^{S} \frac{\pi (x)\log S }{x \log(Sx) \log (x)}   \ \dint x &\leq   \int_{2}^{S} \frac{2\log S }{\log(Sx) \log^2 (x)}    \ \dint x\\
&\leq   \int_{2}^{S} \frac{2}{ \log^2 (x)}  \ \dint x= \left(2\Li(x) - 2\frac{x}{\log x} \right) \Big|_{2}^{S} = 2\Li(S) - \frac{2S}{\log S} + \frac{4}{\log 2},
\end{align*} 
where $\Li(x) = {\displaystyle \int_{2}^x \frac{1}{\log x}} \ \dint x $ is the logarithmic integral. We know from the asymptotic series of $\Li$ about $\infty$ that $\Li(t) - \dfrac{t}{\log t} = \bigO{\dfrac{t}{\log^2 t}}$, and hence, 
$$  \int_{2}^{S} \frac{\pi (x)\log S }{x \log(Sx) \log (x)}   \ \dint x  = \bigO{ \dfrac{S}{\log^2 S}}.$$
Lastly, as a consequence of the Prime Number Theorem, $\pi(S) = \dfrac{S}{\log S} + \bigO{\dfrac{S}{\log^2{S}}}$, and so
\begin{align*}
 \frac{1}{2} \exp\left[  \sum_{i = 1}^{L_r+1} \log \log (p_i S) - \sum_{i = 1}^{L_r+1} \log \log p_i \right] &\leq \frac{1}{2}\exp \left[ (\log 2 )\pi(S) +  \bigO{ \dfrac{S}{\log^2 S}} \right] \\
 &= \frac{1}{2}\exp \left[ (\log 2) \frac{ S}{\log{S}} +  \bigO{ \dfrac{S}{\log^2 S}} \right] 
\end{align*}
as desired. 
\end{proof}
\begin{theorem}
For large $r$, 
$$C_r \leq \frac{1}{2}\exp\left[\frac{1}{2}\dfrac{r^{20/9}}{(\log r)^{29/9}} \left( 1 + \frac{\log\log r}{\log r - \log\log r} + \frac{\bigO{1}}{\log r}\right) \right].$$
\end{theorem}
\begin{proof}
As we showed in the first section, for any $0< w < 11/9$, 
$$p_{L_r}  = P_r \leq \left(\dfrac{r}{w \log r}\right)^{20/9}$$ for large enough $r$. Since ratios of consecutive primes go to $1$, it is also true that for any $0< w < 11/9$, 
$$p_{L_r+1} \leq \left(\dfrac{r}{w \log r}\right)^{20/9}$$ for large enough $r$. We will apply this to Theorem \ref{crformula} using the estimate from Theorem \ref{pampampam}. Let $S =  \left(\dfrac{r}{w \log r}\right)^{20/9}$ for a fixed $0 < w < 11/9$. Then
\begin{align*}
\log (2) S \left(\frac{1}{\log S} + \frac{\bigO{1}}{\log^2 S} \right)&=   \frac{\log 2}{w^{20/9}}\dfrac{r^{20/9}}{(\log r)^{20/9}} \left( \frac{1}{20/9(\log r - \log\log r)} + \frac{\bigO{1}}{(\log r)^2}\right) \\
&< \left( \frac{9\log 2}{20w^{20/9}} \right) \dfrac{r^{20/9}}{(\log r)^{20/9}} \left( \frac{1}{\log r - \log\log r} + \frac{\bigO{1}}{(\log r)^2}\right).
\end{align*}
Since $ \dfrac{9\log 2}{20(11/9)^{20/9}} <1/2$, we can deduce that for large $r$, 
$$C_r \leq \frac{1}{2}\exp\left[\frac{1}{2}\dfrac{r^{20/9}}{(\log r)^{29/9}} \left( 1 + \frac{\log\log r}{\log r - \log\log r} + \frac{\bigO{1}}{\log r}\right) \right]$$
as desired. 
\end{proof}

\section{$C_r$ is Never Equal to $4$ }

In this section we show that $C_r$ does not take all integer values. We do this by proving the following theorem:
\begin{theorem}\label{s4main}
Let $r \in \R, r \geq 1$, and let $C_r$ be the number of connected components of $\sigbar$. Then $C_r \neq 4$.
\end{theorem}

\begin{defn}
For each prime $p$, let $r_p:=\inf\{s\in\R_{>1} : p\text{ is }s-\text{mighty}\}$. 
\end{defn}\label{rpexists}

\begin{lemma}\label{rpexists}
For a prime $p$, $r_p$ is finite, and $p$ is $s$-mighty if and only if $s>r_p$. 
\end{lemma}
\begin{proof}
Recall that a prime $p_m$ is $r$-mighty if 
$$1 + \dfrac{1}{p_m^r} > \prod\limits_{t = m+1}^\infty \dfrac{1}{1 - p_t^{-r}} = \sum\limits_{k \in \N_m} \dfrac{1}{k^r},$$
or equivalently, 
\begin{align}\label{carrots}
\dfrac{1}{p_m^r} > \sum\limits_{k \in \N_m\setminus \{1\}} \dfrac{1}{k^r},
\end{align}
where $\N_m$ is again the set of all positive integers $k$ such that all prime factors of $k$ are greater than $p_m$.
Choose $r>1$, and let $s=r+\delta$ for some $\delta>0$. If we replace $r$ with $s$ in (\ref{carrots}), then the left-hand side multiplies by 
$$\frac{1/p_m^s}{1/p_m^r}=\frac{1}{p_m^\delta}$$ while each term in the sum on the right-hand side of (\ref{carrots}) multiplies by at most 
$$\frac{1/p_{m+1}^s}{1/p_{m+1}^r}=\frac{1}{p_{m+1}^\delta}.$$
Thus the expression on the left of (\ref{carrots}) grows $1 + \Omega(1)$ times faster than the expression on the right, which immediately implies the statement of the lemma. (Here $1 + \Omega(1)$ means that there exists $\delta > 0$ such that the statement is true for $1 + \delta$.))
\end{proof}                                     
\begin{corollary}
The number $N_r$ of $r$-mighty primes is a non-decreasing function of $r$.
\end{corollary} 
We now show that the number of connected components of $\overline{\sigma_{-r}(\N)}$ is never equal to $4$. 
\begin{theorem}[\cite{primeinint}]\label{primeint}
For $n > 25$, there is always a prime between $n$ and $6n/5$.
\end{theorem}
\begin{lemma}\label{3notf}
If $p_m\geq 29$, then $p_m$ is not $3$-mighty. 
\end{lemma}
\begin{proof}
Recall from Lemma \ref{rephrased} that it is sufficient to show that for $p_m \geq 29$, 
$$\dfrac{1}{p_m^3 + 1} < S_{1, m}(r) - S_{1, m}(r)^2/2,$$
where $S_{1, m}(r) = \dfrac{1}{p_{m+1}^3} + \dfrac{1}{p_{m+2}^3} + \cdots.$
Due to Theorem \ref{primeint}, 
$$S_{1, m}(r)= \dfrac{1}{p_{m+1}^3} + \dfrac{1}{p_{m+2}^3} + \cdots \geq \dfrac{1}{p_m^3(6/5)^3} + \dfrac{1}{p_{m}^3(6/5)^6} + \cdots = \dfrac{1}{p_m^3} \dfrac{(5/6)^3}{1 - (5/6)^3} > \dfrac{1.35}{p_m^3}.$$
Because $S_{1, m}(r) < \zeta(3) - 1 < 1$ and $x- x^2/2$ grows on $(0, 1)$, we can deduce 
$$S_{1, m}(r) - S_{1, m}(r)^2/2 \geq  \dfrac{1.35}{p_m^3} -  \dfrac{(1.35)^2}{2 p_m^6} \geq   \dfrac{1.35}{p_m^3} -  \dfrac{(1.35)^2}{2 (29)^3 p_m^3} >  \dfrac{1.3}{p_m^3} >\dfrac{1}{p_m^3+1} .\qedhere$$
\end{proof}

Lemma \ref{3notf} allows us to compute $r$-mighty primes for  $r < 3$ by computer checking all the primes up to $29$. This computer check shows that 
\begin{align}\label{orderofrs}
r_3 < r_2 <r_5 <r_7< r_{p} \text{ for any other }p > 7,
\end{align}
and for $r = 3$ the mighty primes are $2, 3, 5$.
We now apply Sanna's algorithm, which we mentioned earlier, to prove two Lemmas which imply Theorem \ref{s4main}.
\begin{lemma}
Suppose for some $r$, $L_r = 2$. Then $\sigbar$ consists of at most three connected components. 
\end{lemma}
\begin{proof}
We apply Sanna's algorithm.
\begin{enumerate}
\item By Theorem \ref{bananafish}, $ \overline{\sigma_{-r}(\N_{2})} = \left[1,  \prod\limits_{j = 3}^\infty \dfrac{1}{1 - p_t^{-r}}\right].$
\item Now we find the smallest $t \in \Z_{\geq 0}$ such that 
\begin{align}\label{fits}
\sigma_{-r}(3^{t+1})/\sigma_{-r}(3^{t}) < \prod\limits_{j = 3}^\infty \dfrac{1}{1 - p_t^{-r}}.
\end{align} Since $3$ is $r$-mighty, by definition $t = 0$ doesn't satisfy (\ref{fits}). Because $L_r = 2$, we know from \ref{orderofrs} that $1.8 <r_3 \leq r < r_5<2.3$. Using that, an easy computer calculation shows $t = 1$ satisfies (\ref{fits}).
\item Hence
$$\overline{\sigma_{-r}(\N_{1})} = \left[1,  \prod\limits_{j = 3}^\infty \dfrac{1}{1 - p_t^{-r}}\right]\cup \left[1 + \dfrac{1}{3^r},  \prod\limits_{j = 2}^\infty \dfrac{1}{1 - p_t^{-r}}\right].$$ 
\item  Now we find the smallest $t \in \Z_{\geq 0}$ such that $\sigma_{-r}(2^{t+1})/\sigma_{-r}(2^{t}) < \prod\limits_{j = 3}^\infty \dfrac{1}{1 - p_t^{-r}}.$ We can again do it with a simple computation. Using that $1.8 < r < 2.3$, we find that $t = 1$.

\item Now we find the smallest $t$ such that $$\frac{\sigma_{-r}(2^{t+1})}{\sigma_{-r}(2^{t})} <  \frac{\prod\limits_{j = 2}^\infty \dfrac{1}{1 - p_t^{-r}}}{ \left(1 + \dfrac{1}{3^r}\right)} .$$ We compute that $t=1$.\\
\item Hence $\overline{\sigma_{-r}(\N_{0})} $ is equal to
\begin{align*}
 \left[1,  \prod\limits_{j = 3}^\infty \dfrac{1}{1 - p_t^{-r}}\right]\cup \left[1 + \dfrac{1}{3^r},  \prod\limits_{j = 2}^\infty \dfrac{1}{1 - p_t^{-r}}\right] &\cup \left[1 + \frac{1}{2^r},  \frac{1}{1 - 2^{-r}}\prod\limits_{j= 3}^\infty \dfrac{1}{1 - p_t^{-r}}\right]\cup \\
&\cup \left[\left(1 + \dfrac{1}{2^r}\right)\left(1 + \dfrac{1}{3^r}\right), \zeta(r) \right].
\end{align*}
\item We find using Mathematica that for $r < 2.3$, 
$$\left(1 + \dfrac{1}{2^r}\right)\left(1 + \dfrac{1}{3^r}\right) < \zeta(r) (1 - 3^{-r}) =  \frac{1}{1 - 2^{-r}}\prod\limits_{j = 3}^\infty \dfrac{1}{1 - p_t^{-r}} ,$$ and hence 
$$\overline{\sigma_{-r}(\N_{0})}  = \left[1,  \prod\limits_{j = 3}^\infty \dfrac{1}{1 - p_t^{-r}}\right]\cup \left[1 + \dfrac{1}{3^r},  \prod\limits_{j = 2}^\infty \dfrac{1}{1 - p_t^{-r}}\right] \cup \left[1 + \frac{1}{2^r}, \zeta(r) \right],$$
which is at most three disjoint intervals.  \qedhere
\end{enumerate}
\end{proof}

\begin{lemma}
Suppose for some $r$ there are $3$ or more $r$-mighty primes. Then $\sigbar$ has at least $5$ connected components. 
\end{lemma}
\begin{proof}
As previously, let $N_r$ and $C_r$ be the number of $r$-mighty primes and the number of connected components of $\overline{\sigma_{-r}(\mathbb N)}$, respectively. If $N_r\geq 4$, then $C_r\geq 5$ by Theorem \ref{mingaps}. Thus, we may assume $N_r = 3$. The $r$-mighty primes must be $2,3,$ and $5$. Moreover, $r_5\leq  r < r_7$. Using Mathematica, we see that this implies that $2.2\leq r\leq 2.5$. 
For simplicity, let $$u_m=\prod_{t=m+1}^\infty\frac{1}{1-p_t^{-r}}$$ for $m \in \N$. Recall from the proof of Lemma \ref{mingaps} that since $2, 3, 5$ are $r$-mighty, $\overline{\sigma_{-r}(\N)}$ is guaranteed to have gaps
\begin{align}\label{Colin's Equation}
(u_3,\sigma_{-r}(5)),(u_2,\sigma_{-r}(3)),\text{ and }(u_1,\sigma_{-r}(2))
\end{align} 
(recall that a gap of $\overline{\sigma_{-r}(\N)}$ is a bounded connected component of $\R\setminus\overline{\sigma_{-r}(\N)}$).  To complete the proof, we will show that $(u_3\sigma_{-r}(2),\sigma_{-r}(10))$ is another gap of $\overline{\sigma_{-r}(\mathbb N)}$. 

First, note that $\sigma_{-r}(2)<u_3\sigma_{-r}(2)<\sigma_{-r}(5)\sigma_r{-r}(2)=\sigma_{-r}(10)$. This implies that $(u_3\sigma_{-r}(2),\sigma_{-r}(10))$ is a nonempty interval that is disjoint from the three gaps listed in \eqref{Colin's Equation}.  We also note that $u_3\sigma_{-r}(2)$ and $\sigma_{-r}(10)$ are elements of $\overline{\sigma_{-r}(\mathbb N)}$. Thus, we are left to show that $\sigma_{-r}(\mathbb N)\cap(u_3\sigma_{-r}(2),\sigma_{-r}(10))=\emptyset$. 

Choose a positive integer $n$. We will show that $\sigma_{-r}(n)\not\in(u_3\sigma_{-r}(2),\sigma_{-r}(10))$. If $n$ is odd, then it follows from the argument used in the proof of Lemma 3.1 that $\sigma_{-r}(n)<u_1<u_3\sigma_{-r}(2)$. Thus, we may assume $n$ is even. Because $2.2\leq r\leq 2.5$, it is easy to check that $\sigma_{-r}(10)<\sigma_{-r}(4)$. If $4\mid n$, then $\sigma_{-r}(n)\geq\sigma_{-r}(4)>\sigma_{-r}(10)$. Therefore, we may assume $n=2k$ for some \emph{odd} positive integer $k$. If $3\mid k$ or $5\mid k$, then $\sigma_{-r}(n)\geq\sigma_{-r}(10)$. Consequently, we may assume $k$ is not divisible by $2,3,$ or $5$. It follows from the proof of Lemma 3.1 that $\sigma_{-r}(k)<u_3$. Thus, $\sigma_{-r}(n)<u_3\sigma_{-r}(2)$ as desired.  
\end{proof}
This concludes the proof of Theorem \ref{s4main}.

\section{Acknowledgments}
 The research was conducted during the Undergraduate Mathematics Research Program at University of Minnesota Duluth, and supported by the grant NSF / DMS-1659047.
I would like to thank Joe Gallian for creating an amazing working environment in the Duluth REU program, for his invaluable support and inexhaustible good humor. I would like to thank Colin Defant, Levent Alpoge and Mitchell Lee for very thorough and helpful comments and suggestions. I would also like to thank Benjamin Gunby for superb advising. 
\bibliography{SigmaDivisor}

\begin{thebibliography}{10}

\bibitem{Eiref}
Milton Abramowitz and Irene~A Stegun.
\newblock {\em Handbook of mathematical functions: with formulas, graphs, and
  mathematical tables}, volume~55.
\newblock Courier Corporation, 1964.

\bibitem{Colinpaper}
Colin Defant.
\newblock {On the density of ranges of generalized divisor functions}.
\newblock {\em Notes on Number Theory and Discrete Mathematics}, 21:87--88,
  2015.

\bibitem{defantComplex}
Colin Defant.
\newblock Complex divisor functions.
\newblock {\em Geometry and Number Theory}, 1:22--48, 2016.

\bibitem{Colinnew}
Colin Defant.
\newblock Connected components of complex divisor functions.
\newblock {\em arXiv preprint arXiv:1711.04244}, 2017.

\bibitem{difference}
David~R Heath-Brown and Henryk Iwaniec.
\newblock On the difference between consecutive primes.
\newblock {\em Inventiones mathematicae}, 55(1):49--69, 1979.

\bibitem{laatsch}
Richard Laatsch.
\newblock Measuring the abundancy of integers.
\newblock {\em Mathematics Magazine}, 59(2):84--92, 1986.

\bibitem{primeinint}
Jitsuro Nagura.
\newblock On the interval containing at least one prime number.
\newblock {\em Proceedings of the Japan Academy}, 28(4):177--181, 1952.

\bibitem{piconstbound}
J~Barkley Rosser, Lowell Schoenfeld, et~al.
\newblock Approximate formulas for some functions of prime numbers.
\newblock {\em Illinois Journal of Mathematics}, 6(1):64--94, 1962.

\bibitem{Sannapaper}
Carlo Sanna.
\newblock {On the closure of the image of the generalized divisor function}.
\newblock {\em Uniform Distribution Theory}, 12(2):77--90, 2017.

\bibitem{weiner}
Paul~A Weiner.
\newblock The abundancy ratio, a measure of perfection.
\newblock {\em Mathematics Magazine}, 73(4):307--310, 2000.

\end{thebibliography}
\bibliographystyle{plain}
\end{document}